\documentclass[11pt]{article}
\usepackage{amssymb,amsmath,amsthm,amsfonts}
\usepackage{graphicx}
\usepackage{epsfig}

\usepackage{amssymb,amsmath,amsthm,amsfonts,mathrsfs, bbm,dsfont}

\theoremstyle{plain}

\theoremstyle{definition}
\newtheorem{definition}{Definition}[section]

\newtheorem{remark}{Remark}[section]

\newtheorem{theorem}{Theorem}[section]
\newtheorem{lemma}{Lemma}[section]

\numberwithin{equation}{section}

\begin{document}
\openup 1.0\jot
\title{\Large\bf $C^*$-index of observable algebra in the field algebra  determined by \\ a normal group \thanks{This work is supported
by National Science Foundation of China (10971011,11371222)} }
\author{ Xin Qiaoling, Jiang Lining \thanks{E-mail address: jianglining@bit.edu.cn}}
\date{}
\maketitle\begin{center}
\begin{minipage}{16cm}
{\small \it School of Mathematics and Statistics, Beijing Institute
of Technology, Beijing 100081, P. R. China}
\end{minipage}
\end{center}
\vspace{0.05cm}
\begin{center}
\begin{minipage}{16cm}
{\small {\bf Abstract}: Let $G$ be a finite group and $H$ a normal subgroup. $D(H;G)$ is the crossed product of $C(H)$ and ${\Bbb C}G$ which is only a subalgebra of $D(G)$, the quantum double of $G$. One can construct a $C^*$-subalgebra ${\mathcal{F}}_{_H}$ of the field algebra $\mathcal{F}$ of $G$-spin models, such that ${\mathcal{F}}_{_H}$ is a $D(H;G)$-module algebra. The concrete construction of $D(H;G)$-invariant subalgebra ${\mathcal{A}}_{_{(H,G)}}$ of ${\mathcal{F}}_{_H}$ is given. By constructing the quasi-basis of conditional expectation $z_{_H}$ of
${\mathcal{F}}_{_H}$ onto ${\mathcal{A}}_{_{(H,G)}}$, the $C^*$-index of $z_{_H}$ is given.}
\endabstract
\end{minipage}\vspace{0.10cm}
\begin{minipage}{16cm}
\end{minipage}
\end{center}
\begin{center} \vspace{0.01cm}
\end{center}








\section{Introduction }
Assume that $G$ is a finite group with a unit $e$. The $G$-valued spin configuration on the two-dimensional square lattices is the map $\sigma:{\Bbb Z}^2\rightarrow G$ with Euclidean action functional: $$S(\sigma)=\sum \limits_{(x,y)} f(\sigma_x^{-1}\sigma_y),$$ in which the summation runs over the nearest neighbour pairs in ${\Bbb Z}^2$ and $f:G\rightarrow {\Bbb R}$ is a function of the positive type. This kind of classical statistical systems or the corresponding quantum field theories are called $G$-spin models, see \cite{S.Dop,V.F.R.J, K.Szl}. Such models provide the simplest examples of lattice field theories exhibiting quantum symmetry. In general, $G$-spin models with an Abelian group $G$ are known to have a symmetry group $G\times \widehat{G}$, where $\widehat{G}$ is the group of characters of $G$, namely the Pontryagin dual of $G$. If $G$ is non-Abelian, the Pontryagin dual loses its meaning, and the models have a symmetry of a quantum double
 $D(G)$ ({\cite{K.A.Da,G.Mas}}). Here $D(G)$ is defined as the crossed product of $C(G)$, the algebra of complex functions on $G$, and group algebra ${\Bbb C}G$ with respect to the adjoint action of the latter on the former. Then $D(G)$ is a Hopf *-algebra of finite dimension (\cite{P.Ban,C.Kas,F.Nil}). Also as in the traditional quantum field theory, one can define a field algebra ${\mathcal{F}}$ associated with this model (\cite{L.N.Ji}). There is a natural action $\gamma$ of $D(G)$ on ${\mathcal{F}}$ such that ${\mathcal{F}}$ is a $D(G)$-module algebra with respect to the map $\gamma$. Namely, there is a bilinear map $\gamma: D(G)\times{\mathcal{F}}\rightarrow{\mathcal{F}}$ satisfying: $\forall \ a, b\in D(G), F_1, F_2, F\in{\mathcal{F}}$,
\begin{eqnarray*}
    \begin{array}{rcl}
(ab)(F)&=&a(b(F)),\\
a(F_1F_2)&=&\sum\limits_{(a)}a_{(1)}(F_1)a_{(2)}(F_2),\\
a(F^*)&=&(S(a^*)(F))^*.
 \end{array}
\end{eqnarray*}
Here and from now on, by $a(F)$ we always denote $\gamma(a\times F)$ in $\mathcal{F}$. Under the action of $\gamma$ on ${\mathcal{F}}$, the observable algebra ${\mathcal{A}}_{_{(G,G)}}$ as the $D(G)$-invariant subalgebra of $\mathcal{F}$ is obtained. And there exists a duality between ${\mathcal{A}}_G$ and $D(G)$, i.e., there is a unique $C^*$-representation of $D(G)$ such that $D(G)$ and ${\mathcal{A}}_{(G,G)}$ are commutants with each other.

In \cite{Qiao}, we consider a more general situation. Let $H$ be a normal subgroup of $G$, then $D(H;G)$ is defined as the crossed product of $C(H)$ and ${\Bbb C}G$ with respect to the adjoint action of the latter on the former. One can construct a $C^*$-subalgebra ${\mathcal{F}}_{_H}$ of the field algebra $\mathcal{F}$ of $G$-spin models, called the field algebra of $G$-spin models determined by $H$, such that ${\mathcal{F}}_{_H}$ is a $D(H;G)$-module algebra even though $D(H;G)$ is not a Hopf subalgebra of $D(G)$. Then the observable algebra ${\mathcal{A}}_{_{(H,G)}}$, which is the set of fixed points of ${\mathcal{F}}_{_H}$ under the action of $D(H;G)$ is given. Also there exists a duality between $D(H;G)$ and ${\mathcal{A}}_{(H,G)}$.

In this paper, we continue to discuss this model. In Section 2, we find algebraic generators for ${\mathcal{A}}_{_{(H,G)}}$ by means of discussing the local net structure to ${\mathcal{A}}_{_{(H,G)}}$. In Section 3, we construct a quasi-basis for
the conditional expectation $z_{_H}:{\mathcal{F}}_{_H}\rightarrow {\mathcal{A}}_{_{(H,G)}}$, and then obtain the corresponding $C^*$-index ${\mathrm{Index}}\ z_{_H}=|G||H|I$, where $|G|$ and $|H|$ denote the order of the group $G$ and $H$, respectively.

 Throughout this paper, all algebras are complex unital associative algebras. For more
details on Hopf algebras one can refer to the books of Sweedler \cite{M.E.Sw} and Abe \cite{E.Abe}. We shall adopt its notation, such as $S$, $\bigtriangleup$, $\varepsilon$ for the antipode, the comultiplication and the counit, respectively. Also we shall use the summation convention, which is standard in Hopf algebra theory:
\begin{eqnarray*}
    \begin{array}{rcl}
  \bigtriangleup(a)&=&\sum\limits_{(a)}a_{(1)}\otimes a_{(2)},\\
  \bigtriangleup^{(2)}(a)&=&\bigtriangleup\circ({\mathrm{id}}\otimes\bigtriangleup)(a)=\sum\limits_{(a)}a_{(1)}\otimes a_{(2)}\otimes a_{(3)},\\
  \bigtriangleup^{(n)}(a)&=&\bigtriangleup^{(n-1)}\circ({\mathrm{id}}\otimes\bigtriangleup)(a)=\sum\limits_{(a)}a_{(1)}\otimes a_{(2)}\otimes\cdots\otimes a_{(n+1)},
\end{array}
\end{eqnarray*}
where the second one holds since $\bigtriangleup\circ(\mathrm{id}\otimes\bigtriangleup)=\bigtriangleup\circ(\bigtriangleup\otimes {\mathrm{id}})$, and so on.
\section{The structure of the observable algebra in ${\mathcal{F}}_{_H}$}
Suppose that $H$ is a normal subgroup of $G$. In the previous paper  \cite{Qiao}, we defined a Hopf *-algebra $D(H;G)$ and then constructed a $C^*$-subalgebra ${\mathcal{F}}_{_H}$ in the field algebra $\mathcal{F}$ of $G$-spin models. Under the action $\gamma$ of $D(H;G)$ on it, ${\mathcal{F}}_{_H}$ becomes a $D(H;G)$-module algebra and the observable algebra $A_{_{(H;G)}}$ as the $D(H;G)$-invariant subalgebra of ${\mathcal{F}}_{_H}$ is given. This section will discuss a local net structure to $A_{_{(H;G)}}$, which can be achieved by finding algebraic generators for $A_{_{(H;G)}}$ with local commutation relations.  Let us begin with the following definition.

\begin{definition}$^{\cite{Qiao}}$
$D(H;G)$ is the crossed product of $C(H)$ and group algebra ${\Bbb C}G$, where $C(H)$ denotes the set of complex functions on $H$, with respect to the adjoint action of the latter on the former.
\end{definition}

Using the linear basis elements $(h,g)$ of $D(H;G)$, the multiplication can be written as:
$$(h_1,g_1)(h_2,g_2)=\delta_{h_1g_1,g_1h_2}(h_1,g_1g_2).$$

Clearly, $\sum\limits_{h\in H}(h,e)$ is the unit of $D(H;G)$. Also, the structure maps are defined as
\begin{eqnarray*}
 \begin{array}{rcll}
{(h,g)}^*&=&(g^{-1}hg,g^{-1}),&(\mbox{*-operation})\\
\bigtriangleup(h,g)&=&\sum \limits_{t\in H}(t,g)\otimes (t^{-1}h,g),&(\mbox{coproduct}) \\
\varepsilon(h,g)&=&\delta_{h,e},&(\mbox{counit})\\
S(h,g)&=&(g^{-1}h^{-1}g,g^{-1}),&(\mbox{antipode})
 \end{array}
\end{eqnarray*}
where $h\in H, g\in G$ and $\delta_{g,h}=\left\{\begin{array}{cc}
                          1, & {\mathrm{if}} \ \  g=h \\
                          0, & {\mathrm{if}} \ \ g\neq h
                        \end{array}\right.$.
One can prove $D(H;G)$ is a Hopf *-algebra with a unique element \begin{eqnarray*}
    \begin{array}{c}
    z_{_H}=\frac{1}{|G|}\sum\limits_{g\in G}(e,g),
          \end{array}
\end{eqnarray*}
 called a cointegral, satisfying
  $$az_{_H}=z_{_H}a=\varepsilon(a)z_{_H}, \ \ \forall a\in D(H;G),$$
  and $\varepsilon(z_{_H})=1$. As a result, $D(H;G)$ is a semisimple *-algebra of finite dimension. Consequently it can be a $C^*$-algebra in a natural way ({\cite{Qiao}}).

\begin{remark}
 (1) If $H$ is a subgroup of $G$, not a normal subgroup. One can prove there is not the adjoint action of ${\Bbb C}H$ on $C(G)$, and then $D(H;G)$ can not be defined.

(2) Different from the case of $D(G;H)$, which is the crossed product of $C(G)$ and ${\Bbb C}H$ with respect to the adjoint action of the latter on the former (\cite{L.N.Ji,L.N.Jia}), $D(H;G)$ is not a Hopf subalgebra of $D(G)$, even though it is a subalgebra of $D(G)$.

 (3) Also, the relation $S^2=\mathrm{id}$ holds in $D(H;G)$, which implies that $\forall a\in D(H;G),$
 \begin{eqnarray*}
    \begin{array}{c}
    \sum\limits_{(a)}S(a_{(2)})a_{(1)}=
    \sum\limits_{(a)}a_{(2)}S(a_{(1)})=\varepsilon(a)1_{D(H;G)}.
      \end{array}
\end{eqnarray*}
\end{remark}

As in the traditional case, one can define the local quantum field algebra associated with the model.

\begin{definition}$^{\cite{Qiao}}$
 ${\mathcal{F}}_{_{H,loc}}$ is an associative algebra with a unit $I$ generated by
 $\{\delta_g(x), \rho_h(l): g\in G, h\in H; x\in {\Bbb Z}, l\in {\Bbb Z}+\frac{1}{2}\}$ subject to
 \begin{eqnarray*}
    \begin{array}{rcl}
\sum\limits_{g\in G}\delta_g(x)&=&I=\rho_e(l),\\
\delta_{g_1}(x)\delta_{g_2}(x)&=&\delta_{g_1,g_2}\delta_{g_1}(x),\\
 \rho_{h_1}(l)\rho_{h_2}(l)&=&\rho_{h_1h_2}(l),\\
\delta_{g_1}(x)\delta_{g_2}(x')&=&\delta_{g_2}(x')\delta_{g_1}(x),\\
 \rho_h(l)\delta_g(x)&=& \left\{\begin{array}{cc}
                          \delta_{hg}(x)\rho_h(l), & l<x, \\
                          \delta_g(x)\rho_h(l), & l>x,
                        \end{array}\right.\\
 \rho_{h_1}(l)\rho_{h_2}(l')&=&
                          \rho_{h_2}(l')\rho_{{h_2}^{-1}h_1{h_2}}(l), \ \  l>l'.
  \end{array}
\end{eqnarray*}
for $x, x' \in {\Bbb Z};\ l, l'\in {\Bbb Z}+\frac{1}{2} \ {\mathrm{and}} \ h_1, h_2\in H, g_1,g_2\in G$.
In particular, if $H=G$, by ${\mathcal{F}}_{_{loc}}$ we denote ${\mathcal{F}}_{_{H,loc}}$.
\end{definition}

The *-operation is defined on the generators as $\delta_g^\ast(x)=\delta_g(x), \ \rho_h^\ast(l)=\rho_{h^{-1}}(l)$ and is extended antilinearly and antimultiplicatively to ${\mathcal{F}}_{_{H,loc}}$. In this way, ${\mathcal{F}}_{_{H,loc}}$ becomes a unital *-algebra.

For any finite subset $\Lambda\subseteq \frac{1}{2}\Bbb{Z}$, let
${\mathcal{F}}_{_{H}}(\Lambda)$ be the subalgebra of ${\mathcal{F}}_{_{H,loc}}$ generated by
 \begin{eqnarray*}
\{\delta_g(x), \rho_h(l):
 g\in G, h\in H, x, l\in\Lambda\}.
 \end{eqnarray*}
In particular, we consider an increasing sequence of intervals $\Lambda_n, n\in\Bbb{N}$, where
\begin{eqnarray*}
    \begin{array}{rcl}
    \Lambda_{2n}&=&\{s\in\frac{1}{2}\Bbb{Z}: -n+\frac{1}{2}\leq s\leq n\}\\
    \Lambda_{2n+1}&=&\{s\in\frac{1}{2}\Bbb{Z}: -n-\frac{1}{2}\leq s\leq n\}
\end{array}
\end{eqnarray*}
In \cite{K.Szl}, the authors have shown that ${\mathcal{F}}(\Lambda_n), n\in\Bbb{N}$ are full matrix algebras, they can be identified with $M_{|G|^n}$. Moreover, under the induced norm, ${\mathcal{F}}(\Lambda_n)$ are finite dimensional $C^*$-algebras.
Hence ${\mathcal{F}}_{_{H}}(\Lambda_n), n\in\Bbb{N}$ are subalgebras of full matrix algebras, and then they are finite dimensional $C^*$-algebras. The natural embeddings $\iota_n: {\mathcal{F}}_{_{H}}(\Lambda_n)\rightarrow{\mathcal{F}}_{_{H}}
(\Lambda_{n+1})$, that identify the $\delta$ and $\rho$ generators, are norm preserving.
Using the $C^*$-inductive limit (\cite{B.R.Li}), a $C^*$-algebra ${\mathcal{F}}_{_H}$ can be given by
\begin{eqnarray*}
    \begin{array}{c}
    {\mathcal{F}}_{_H}=\overline{\bigcup\limits_{n}
{\mathcal{F}}_{_H}(\Lambda_{n})},
 \end{array}
\end{eqnarray*}
called the field algebra of $G$-spin models determined by a normal subgroup $H$.\\

There is an action $\gamma$ of $D(H;G)$ on ${\mathcal{F}}_{_H}$ in the following. For $x\in {\Bbb Z}; \ l\in {\Bbb Z}+\frac{1}{2} \ {\mathrm{and}} \ h\in H, g\in G$, set
\begin{eqnarray*}
    \begin{array}{rcl}
(h,g)\delta_f(x)&=&\delta_{h,e}\delta_{gf}(x), \ \ \forall f\in G,\\
(h,g)\rho_t(l)&=&\delta_{h,gtg^{-1}}\rho_h(l), \ \ \forall t\in H.
 \end{array}
\end{eqnarray*}
Then the $\gamma$ can be extended continuously to an action of  $D(H;G)$ on ${\mathcal{F}}_{_H}$, such that ${\mathcal{F}}_{_H}$ is a $D(H;G)$-module algebra with respect to the $\gamma$. Namely, the $\gamma$ satisfies the following relations:
\begin{eqnarray*}
    \begin{array}{rcl}
(ab)(F)&=&a(b(F)),\\
a(F_1F_2)&=&\sum\limits_{(a)}a_{(1)}(F_1)a_{(2)}(F_2),\\
a(F^*)&=&(S(a^*)(F))^*
 \end{array}
\end{eqnarray*}
for $ a, b\in D(H;G), \ F_1, F_2, F\in {\mathcal{F}}_{_H}$.

Set $${\mathcal{A}}_{_{(H,G)}}=\{F\in {\mathcal{F}}_{_H}: z_{_H}(F)=F\}.$$
Then ${\mathcal{A}}_{_{(H,G)}}$ is a subalgebra of ${\mathcal{F}}_{_H}$.

\begin{lemma}$^{\cite{Qiao}}$ $z_{_H}:{\mathcal{F}}_{_H}\rightarrow{\mathcal{A}}_{_{(H,G)}}$ satisfies the following conditions:

 $(1)$\ $z_{_H}(I)=I$ where $I$ is the unit of ${\mathcal{F}}_{_H}$;

 $(2)$\ (bimodular property) $\forall \ F_1, F_2\in{\mathcal{A}}_{_{(H,G)}}, \ F\in {\mathcal{F}}_{_H}$,
 $$z_{_H}(F_1FF_2)=F_1z_{_H}(F)F_2;$$

 $(3)$\ $z_{_H}$ is positive.
\end{lemma}

In the following a linear map $\Gamma$ from a unital $C^*$-algebra $B$ onto its unital $C^*$-subalgebra $A$ with properties (1)-(3) in Lemma 2.1 is  called a conditional expectation. If $\Gamma$ is a conditional expectation from $B$ onto $A$, then $\Gamma$ is a projection of norm one (\cite{B.R.Li}). The conditional expectation $z_{_H}:{\mathcal{F}}_{_H}\rightarrow{\mathcal{A}}_{_{(H,G)}}$ will be addressed in the next section.

From the above lemma, one can prove ${\mathcal{A}}_{_{(H,G)}}$ is the $C^*$-subalgebra of ${\mathcal{F}}_{_H}$, called an observable algebra related to $H$ in the field algebra ${\mathcal{F}}$ of $G$-spin models. Moreover, if $H_1\subsetneq H_2$ with $H_i\triangleleft G$ for $i=1,2$, then ${\mathcal{A}}_{_{(H_1,G)}}\subsetneq{\mathcal{A}}_{_{(H_2,G)}}$, since
$\bigtriangleup_{H_1}^{(n)}(z_{H_1})
\leq\bigtriangleup_{H_2}^{(n)}(z_{H_2})$
as projections on
${{\mathcal{F}}_{H_1}}^{\otimes^{n+1}}$ for $n\in {\Bbb N}$.

In this section, we will give the concrete construction of ${\mathcal{A}}_{_{(H,G)}}$. In order to do this, for $g\in G$, $x\in\Bbb{Z}$, and $l\in\Bbb{Z}+\frac{1}{2}$, set
\begin{eqnarray*}
    \begin{array}{rcl}
v_g(x)&=&\sum\limits_{h\in G}\varrho_{hg^{-1}h^{-1}}(x-\frac{1}{2})\delta_h(x)
\varrho_{hgh^{-1}}(x+\frac{1}{2}),\\
w_g(l)&=&\sum\limits_{h\in G}\delta_h(l-\frac{1}{2})\delta_{hg}(l+\frac{1}{2}).
\end{array}
\end{eqnarray*}

\begin{lemma}
Let $\Lambda_{n-\frac{1}{2},m}\subseteq\frac{1}{2}\Bbb{Z}$ be a finite interval for $n,m\in\Bbb{Z}$. The $D(H;G)$-invariant subalgebra of
${\mathcal{F}}_{_H}(\Lambda_{n-\frac{1}{2},m})$ is generated by
\begin{eqnarray*}
    \left\{\omega_g(x), v_h(l):
 g\in G, h\in H, x,l\in\Lambda_{n,m-\frac{1}{2}}\right\}.
\end{eqnarray*}
That is
\begin{eqnarray*}
   z_{_H}\left({\mathcal{F}}_{_H}(\Lambda_{n-\frac{1}{2},m})\right)
   =\left\{\omega_g(x), v_h(l):
 g\in G, h\in H, x,l\in\Lambda_{n,m-\frac{1}{2}}\right\}.
\end{eqnarray*}
\end{lemma}

\begin{proof}
 We know that ${\mathcal{F}}_{_H}(\Lambda_{{\frac{1}{2}},m})$ is a $C^*$-subalgebra of ${\mathcal{F}}_{_{H,loc}}$, generated by
\begin{eqnarray*}
    \left\{\delta_g(x), \varrho_h(l):
 g\in G, h\in H, x,l\in\Lambda_{{\frac{1}{2}},m}\right\}.
\end{eqnarray*}
Notice that for $h_i\in H, \ \ i=1,2,\cdots m$ with $h_1h_2\cdots h_m=e$,
 \begin{eqnarray*}
    \begin{array}{rcl}
&&z_{_H}\left(\delta_{g_1}(1)\delta_{g_2}(2)\cdots\delta_{g_m}(m)
\varrho_{h_1}(\frac{1}{2})\varrho_{h_2}(\frac{3}{2})
\cdots\varrho_{h_m}(m-\frac{1}{2})\right)\\
&=&\frac{1}{|G|}\sum\limits_{f\in G}(f,e)\left(\delta_{g_1}(1)\delta_{g_2}(2)\cdots\delta_{g_m}(m)
\varrho_{h_1}(\frac{1}{2})\varrho_{h_2}(\frac{3}{2})
\cdots\varrho_{h_m}(m-\frac{1}{2})\right)\\
&=&\frac{1}{|G|}\sum\limits_{f\in G}\sum\limits_{t_i\in H}
\Big((t_1,f)\delta_{g_1}(1)(t_1t_2^{-1},f)\delta_{g_2}(2)\cdots
(t_{m-1}t_{m}^{-1},f)\delta_{g_m}(m)\\
&&(t_{m}t_{m+1}^{-1},f)\varrho_{h_1}(\frac{1}{2})
(t_{m+1}t_{m+2}^{-1},f)\varrho_{h_2}(\frac{3}{2})
\cdots(t_{2m-1}^{-1},f)\varrho_{h_m}(m-\frac{1}{2})\Big)\\
&=&\frac{1}{|G|}\sum\limits_{f\in G}
\Big(\delta_{fg_1}(1)\delta_{fg_2}(2)\cdots\delta_{fg_m}(m)
\varrho_{fh_1f^{-1}}(\frac{1}{2})
\varrho_{fh_2f^{-1}}(\frac{3}{2})
\cdots\varrho_{fh_{m-1}^{-1}h_{m-2}^{-1}\cdots h_{1}^{-1}f^{-1}}(m-\frac{1}{2})\Big)\\
&=&\frac{1}{|G|}\sum\limits_{s\in G}
\Big(\delta_{s}(1)\delta_{sg_1^{-1}g_2}(2)\cdots
\delta_{sg_1^{-1}g_m}(m)
\varrho_{sg_1^{-1}h_1g_1s^{-1}}(\frac{1}{2})
\varrho_{sg_1^{-1}h_2g_1s^{-1}}(\frac{3}{2})
\cdots\\
&&\varrho_{sg_1^{-1}h_{m-1}^{-1}h_{m-2}^{-1}\cdots
 h_{1}^{-1}g_1s^{-1}}(m-\frac{1}{2})\Big),
      \end{array}
\end{eqnarray*}
which together with the following equation
 \begin{eqnarray*}
    \begin{array}{rcl}
&&w_{x_1}(\frac{3}{2})w_{x_2}(\frac{5}{2})
\cdots w_{x_{m-1}}(m-\frac{1}{2})v_{y_1}(1)
v_{y_2}(2)\cdots v_{y_{m-1}}(m-1)\\
&=&\sum\limits_{s_i\in G}\sum\limits_{t_i\in G}
\delta_{s_1}(1)\delta_{s_1x_1}(2)
\delta_{s_2}(2)\delta_{s_2x_2}(3)\cdots\delta_{s_{m-1}}(m-1)
\delta_{s_{m-1}x_{m-1}}(m)\rho_{t_1y_1^{-1}t_1^{-1}}(\frac{1}{2})
\delta_{t_1}(1)\rho_{t_1y_1t_1^{-1}}(\frac{3}{2})\\
&&\rho_{t_2y_2^{-1}t_2^{-1}}(\frac{3}{2})\delta_{t_2}(2)
\rho_{t_2y_2t_2^{-1}}(\frac{5}{2})\cdots
\rho_{t_{m-1}y_{m-1}^{-1}t_{m-1}^{-1}}(m-\frac{3}{2})
\delta_{t_{m-1}}({m-1})
\rho_{t_{m-1}y_{m-1}t_{m-1}^{-1}}(m-\frac{1}{2})\\
&=&\sum\limits_{s_i\in G}\sum\limits_{t_i\in G}
\delta_{s_1}(1)\delta _{t_1y_1^{-1}}(1)\delta_{s_1x_1}(2)
\delta_{s_2}(2)\delta_{t_2y_2^{-1}}(1)\cdots
\delta_{s_{m-2}x_{m-2}}(m-1)\delta_{s_{m-1}}({m-1})\\
&&
\delta _{t_{m-1}y_{m-1}^{-1}}(m-1)\delta_{s_{m-1}x_{m-1}}(m)
\rho_{t_1y_1^{-1}t_1^{-1}}(\frac{1}{2})
\rho_{t_1y_1t_1^{-1}t_2y_2^{-1}t_2^{-1}}(\frac{3}{2})
\rho_{t_2y_2t_2^{-1}t_3y_3^{-1}t_3^{-1}}(\frac{5}{2})\cdots\\
&&\rho_{t_{m-2}y_{m-2}^{-1}t_{m-2}^{-1}y_{m-1}^{-1}t_{m-1}^{-1}}
(m-\frac{3}{2})
\rho_{t_{m-1}y_{m-1}t_{m-1}^{-1}}(m-\frac{1}{2})\\
&=&\sum\limits_{s\in G}\delta_{s}(1)\delta_{sx_1}(2)\delta_{sx_1x_2}(3)
\cdots\delta_{sx_1x_2\cdots x_{m-1}}(m)
\rho_{sy_1^{-1}s^{-1}}(\frac{1}{2})
\rho_{sy_1x_1y_2^{-1}x_1^{-1}s^{-1}}(\frac{3}{2})
\rho_{sx_1y_2x_2y_3^{-1}x_2^{-1}x_1^{-1}s^{-1}}(\frac{5}{2})\\
&&\cdots\rho_{sx_1y_2x_2\cdots x_{m-3}y_{m-2}x_{m-2}y_{m-1}^{-1}x_{m-2}^{-1}
\cdots x_1^{-1}s^{-1}}(m-\frac{3}{2})
\rho_{sx_1y_2x_2\cdots x_{m-3}x_{m-2}y_{m-1}x_{m-2}^{-1}
\cdots x_1^{-1}s^{-1}}(m-\frac{1}{2})
      \end{array}
\end{eqnarray*}
yields that
 \begin{eqnarray*}
    \begin{array}{rcl}
&&z_{_H}\left(\delta_{g_1}(1)\delta_{g_2}(2)\cdots\delta_{g_m}(m)
\varrho_{h_1}(\frac{1}{2})\varrho_{h_2}(\frac{3}{2})
\cdots\varrho_{h_m}(m-\frac{1}{2})\right)\\
&=&\frac{1}{|G|}w_{{g_1}^{-1}{g_2}}(\frac{3}{2})
w_{{g_2}^{-1}{g_3}}(\frac{5}{2})
\cdots w_{g_{m-1}^{-1}{g_m}}(m-\frac{1}{2})v_{g_1^{-1}h_1^{-1}g_1}(1)
v_{g_2^{-1}h_2^{-1}h_1^{-1}g_2}(2)\\
&&v_{g_3^{-1}h_3^{-1}h_2^{-1}h_1^{-1}g_3}(3)\cdots
v_{g_{m-1}^{-1}h_{m-1}^{-1}h_{m-1}^{-1}\cdots h_1^{-1}g_{m-1}}(m-1).
      \end{array}
\end{eqnarray*}
Hence, $z_{_H}\left({\mathcal{F}}_{_H}(\Lambda_{{\frac{1}{2}},m})\right)$ is a $C^*$-subalgebra of ${\mathcal{A}}_{_{(H,G)}}$, generated by
\begin{eqnarray*}
    \left\{\omega_g(x), v_h(l):
 g\in G, h\in H, x,l\in\Lambda_{1,m-\frac{1}{2}}\right\}.
\end{eqnarray*}
By induction, one can show $z_{_H}\left({\mathcal{F}}_{_H}(\Lambda_{{n-\frac{1}{2}},m})\right)$
 is  generated by
\begin{eqnarray*}
    \left\{\omega_g(x), v_h(l):
 g\in G, h\in H, x,l\in\Lambda_{n,m-\frac{1}{2}}\right\}.
\end{eqnarray*}
\end{proof}

\begin{remark}For $\Lambda\subseteq\frac{1}{2}\Bbb{Z}$, let
 \begin{eqnarray*}
    \begin{array}{c}
    {\mathcal{A}}_{_H}(\Lambda)=\langle v_h(x), w_g(l): h\in H, g\in G, x,l\in\Lambda\rangle.
          \end{array}
\end{eqnarray*}
Lemma 2.2 together with Lemma 2.1 implies that $z_{_H}:{\mathcal{F}}_{_H}(\Lambda_{{n-\frac{1}{2}},m})
\rightarrow{\mathcal{A}}_{_H}(\Lambda_{n,m-\frac{1}{2}})$ is a conditional expectation.
\end{remark}

\begin{theorem}
The observable algebra ${\mathcal{A}}_{_{(H,G)}}$ is the $C^*$-algebra given by the $C^*$-inductive limit
 \begin{eqnarray*}
    \begin{array}{c}
    {\mathcal{A}}_{_{(H,G)}}=
\overline{\bigcup\limits_{\Lambda}{\mathcal{A}}_{_H}
(\Lambda)}.
   \end{array}
\end{eqnarray*}
\end{theorem}

\begin{proof}
If $A\in {\mathcal{A}}_{_{(H,G)}}$ and $\varepsilon> 0$, then from Lemma 2.2 and the continuity of the projection $z_{_H}$, we know $A=z_{_H}(A)$ and there is $B\in {\mathcal{F}}_{_H}(\Lambda_{n-\frac{1}{2},m})$ with $\|A-B\|<\varepsilon$, which implies that
$$\|A-z_{_H}(B)\|=\|z_{_H}(A-B)\|\leq\|A-B\|<\varepsilon,$$
and $z_{_H}(B)\in {\mathcal{A}}_{_H}(\Lambda_{n,m-\frac{1}{2}})$.
\end{proof}

\section{$C^*$-index}
This section will give the $C^*$-index of conditional expectation $z_{_H}:{\mathcal{F}}_{_H}\rightarrow{\mathcal{A}}_{_{(H,G)}}$, where $H$ is a normal subgroup of $G$ with $[G:H]=k$ and $t_1=e, t_2,\cdots, t_k$ is a left coset representation of $H$ in $G$, namely
$G=\bigcup\limits_{i=1}^kt_iH$ and $i\neq j$ induces that $t_iH\cap t_jH=\emptyset$, where $e$ is the unit of $G$.

\begin{definition}
Let $\Gamma$ be a conditional expectation from a unital $C^*$-algebra $B$ onto its unital $C^*$-subalgebra $A$. A finite family $\{(u_1,u_1^*), (u_2,u_2^*),\cdots, (u_n,u_n^*)\}\subseteq B\times B$ is called a quasi-basis for $\Gamma$ if for all $a\in B$,
 \begin{eqnarray*}
    \begin{array}{c}
\sum\limits_{i=1}^{n}u_i\Gamma(u_i^*a)=a=\sum\limits_{i=1}^{n}\Gamma(a u_i)u_i^*.
    \end{array}
\end{eqnarray*}
Furthermore, if there exists a quasi-basis for $\Gamma$, we call $\Gamma$ of index-finite type. In this case we define the index of $\Gamma$ by
 \begin{eqnarray*}
    \begin{array}{c}
{\mathrm{Index}}\ \Gamma=\sum\limits_{i=1}^nu_iu_i^*.
      \end{array}
\end{eqnarray*}
\end{definition}

\begin{remark}
(1) If $\Gamma$ is a conditional expectation of index-finite type, then ${\mathrm{Index}}\ \Gamma$ is in the center of $A$ and does not depend on the choice of quasi-basis (\cite{Y.Wat}).

(2) Let $N \subseteq M$ be factors of type II$_1$ and $\Gamma: M\rightarrow N$ the canonical conditional expectation determined by the unique normalized trace on $M$, then ${\mathrm{Index}}\ \Gamma$ is exactly Jones index $[M,N]$ based on the coupling constant (\cite{M.PP}). More generally, let $M$ be a ($\sigma$-finite) factor with a
subfactor $N$ and $\Gamma$ a normal conditional expectation from $M$
onto $N$, then $\Gamma$ is of index-finite if and only if ${\mathrm{Index}}\ \Gamma$ is finite in the sence of \cite{H.Ko}, and the values of ${\mathrm{Index}}\ \Gamma$ are equal.
\end{remark}

\begin{theorem}
For fixed $k\in\Bbb{Z}, x\in G, y\in H$, set
$$u_{x,y}=\sqrt{|G|}\delta_x(k)\rho_y(k+\frac{1}{2}).$$
Then $\{(u_{x,y},u_{x,y}^*):x\in G,y\in H\}$ is a quasi-basis of $z_{_H}:{\mathcal{F}}_{_H}\rightarrow{\mathcal{A}}_{_{(H,G)}}$.
\end{theorem}

\begin{proof}
Without loss of generality, one can consider the case $k=1$.

Firstly, one can show that
$\{(u_{x,y},u_{x,y}^*):x\in G,y\in H\}$ is a quasi-basis of $z_{_H}:{\mathcal{F}}_{_H}(\Lambda_{{\frac{1}{2}},m})
\rightarrow{\mathcal{A}}_{_H}(\Lambda_{1,m-\frac{1}{2}})$, for any $m\in\Bbb{Z}$ and $m>1$.

Note that
 \begin{eqnarray*}
    \begin{array}{rcl}
&&\sum\limits_{x\in G}\sum\limits_{y\in H}
u_{x,y}z_{_H}\Big[u_{x,y}^*\delta_{g_1}(1)\delta_{g_2}(2)
\cdots\delta_{g_m}(m)
\varrho_{h_1}(\frac{1}{2})\varrho_{h_2}(\frac{3}{2})
\cdots\varrho_{h_m}(m-\frac{1}{2})\Big]\\
&=&|G|\sum\limits_{x\in G}\sum\limits_{y\in H}u_{x,y}
z_{_H}\Big[\delta_x(1)\varrho_{y^{-1}}(\frac{3}{2})
\delta_{g_1}(1)\delta_{g_2}(2)\cdots\delta_{g_m}(m)
\varrho_{h_1}(\frac{1}{2})\varrho_{h_2}(\frac{3}{2})
\cdots\varrho_{h_m}(m-\frac{1}{2})\Big]\\
&=&|G|\sum\limits_{x\in G}\sum\limits_{y\in H}u_{x,y}
z_{_H}\Big[\delta_x(1)\delta_{g_1}(1)\delta_{y^{-1}g_2}(2)
\delta_{y^{-1}g_3}(3)\cdots\delta_{y^{-1}g_m}(m)
\varrho_{y^{-1}}(\frac{3}{2})
\varrho_{h_1}(\frac{1}{2})\varrho_{h_2}(\frac{3}{2})\\
&&\cdots\varrho_{h_m}(m-\frac{1}{2})\Big]\\
&=&|G|\sum\limits_{x\in G}\sum\limits_{y\in H}u_{x,y}
z_{_H}\Big[\delta_x(1)\delta_{g_1}(1)\delta_{y^{-1}g_2}(2)
\delta_{y^{-1}g_3}(3)\cdots\delta_{y^{-1}g_m}(m)
\varrho_{h_1}(\frac{1}{2})
\varrho_{h_1^{-1}y^{-1}h_1}(\frac{3}{2})\\
&&\varrho_{h_2}(\frac{3}{2})
\cdots\varrho_{h_m}(m-\frac{1}{2})\Big]\\
&=&|G|\sum\limits_{x\in G}\sum\limits_{y\in H}
\delta_x(1)\varrho_{y}(\frac{3}{2})
z_{_H}\Big[\delta_x(1)\delta_{g_1}(1)\delta_{y^{-1}g_2}(2)
\delta_{y^{-1}g_3}(3)\cdots\delta_{y^{-1}g_m}(m)
\varrho_{h_1}(\frac{1}{2})
\varrho_{h_1^{-1}y^{-1}h_1 h_2}(\frac{3}{2})\\
&&\varrho_{h_3}(\frac{5}{2})
\cdots\varrho_{h_m}(m-\frac{1}{2})\Big]\\
&=&|G|\delta_{g_1}(1)\varrho_{h_1h_2\cdots h_m}(\frac{3}{2})
z_{_H}\Big[\delta_{g_1}(1)\delta_{h_m^{-1}h_{m-1}^{-1}\cdots h_1^{-1}g_2}(2)
\delta_{h_m^{-1}h_{m-1}^{-1}\cdots h_1^{-1}g_3}(3)\cdots\delta_{h_m^{-1}h_{m-1}^{-1}\cdots h_1^{-1}g_m}(m)\\
&&\varrho_{h_1}(\frac{1}{2})
\varrho_{h_1^{-1}h_m^{-1}h_{m-1}^{-1}\cdots h_4^{-1}h_3^{-1}}(\frac{3}{2})\varrho_{h_3}(\frac{5}{2})
\cdots\varrho_{h_m}(m-\frac{1}{2})\Big]\\
&=&\delta_{g_1}(1)\varrho_{h_1h_2\cdots h_m}(\frac{3}{2})
\Big[\sum\limits_{s\in G}
\delta_s(1)\delta_{sg_1^{-1}h_m^{-1}h_{m-1}^{-1}\cdots h_1^{-1}g_2}(2)
\delta_{sg_1^{-1}h_m^{-1}h_{m-1}^{-1}\cdots h_1^{-1}g_3}(3)\cdots
\delta_{sg_1^{-1}h_m^{-1}h_{m-1}^{-1}\cdots h_1^{-1}g_m}(m)\\
&&\varrho_{sg_1^{-1}h_1g_1s^{-1}}(\frac{1}{2})
\varrho_{sg_1^{-1}h_1^{-1}h_m^{-1}h_{m-1}^{-1}\cdots h_4^{-1}h_3^{-1} g_1s^{-1}}(\frac{3}{2})\varrho_{sg_1^{-1}h_3g_1s^{-1}}(\frac{5}{2})
\cdots\varrho_{sg_1^{-1}h_mg_1s^{-1}}(m-\frac{1}{2})\Big]\\
&=&\delta_{g_1}(1)\varrho_{h_1h_2\cdots h_m}(\frac{3}{2})
\delta_{h_m^{-1}h_{m-1}^{-1}\cdots h_1^{-1}g_2}(2)
\delta_{h_m^{-1}h_{m-1}^{-1}\cdots h_1^{-1}g_3}(3)\cdots
\delta_{h_m^{-1}h_{m-1}^{-1}\cdots h_1^{-1}g_m}(m)\\
&&\varrho_{h_1}(\frac{1}{2})
\varrho_{h_1^{-1}h_m^{-1}h_{m-1}^{-1}\cdots h_4^{-1}h_3^{-1}}(\frac{3}{2})
\varrho_{h_3}(\frac{5}{2})
\cdots\varrho_{h_m}(m-\frac{1}{2})\\
&=&\delta_{g_1}(1)\delta_{g_2}(2)\cdots\delta_{g_m}(m)
\varrho_{h_1}(\frac{1}{2})
\varrho_{h_1^{-1}h_1h_2\cdots h_mh_1}(\frac{3}{2})
\varrho_{h_1^{-1}h_m^{-1}h_{m-1}^{-1}\cdots h_4^{-1}h_3^{-1}}(\frac{3}{2})
\varrho_{h_3}(\frac{5}{2})
\cdots\varrho_{h_m}(m-\frac{1}{2})\\
&=&\delta_{g_1}(1)\delta_{g_2}(2)\cdots\delta_{g_m}(m)
\varrho_{h_1}(\frac{1}{2})
\varrho_{h_2}(\frac{3}{2})
\varrho_{h_3}(\frac{5}{2})
\cdots\varrho_{h_m}(m-\frac{1}{2})
      \end{array}
\end{eqnarray*}
which yields that for any $a\in {\mathcal{F}}_{_H}(\Lambda_{{\frac{1}{2}},m})$,
 \begin{eqnarray*}
    \begin{array}{c}
    \sum\limits_{x\in G}\sum\limits_{y\in H}
    u_{x,y}z_{_H}(u_{x,y}^*a)=a.
      \end{array}
\end{eqnarray*}
Similarly, one can verify
 \begin{eqnarray*}
    \begin{array}{c}
    \sum\limits_{x\in G}\sum\limits_{y\in H}z_{_H}(a u_{x,y})u_{x,y}^*=a, \ \ \forall a\in {\mathcal{F}}_{_H}(\Lambda_{{\frac{1}{2}},m}).
      \end{array}
\end{eqnarray*}

By induction, we can show that $\{(u_{x,y},u_{x,y}^*):x\in G,y\in H\}$ is a quasi-basis of $z_{_H}:{\mathcal{F}}_{_H}(\Lambda_{{n-\frac{1}{2}},m})
\rightarrow{\mathcal{A}}_{_H}(\Lambda_{n,m-\frac{1}{2}})$, for any $n,m\in\Bbb{Z}$ and $n<m$.

Since $z_{_H}$ is a projection of norm one, $z_{_H}$
can therefore be extended to the map of $\overline{\bigcup\limits_{n<m}
{\mathcal{F}}_{_H}(\Lambda_{{n-\frac{1}{2}},m})}$ onto $\overline{\bigcup\limits_{n<m}{\mathcal{A}}_{_H}
(\Lambda_{n,m-\frac{1}{2}})}$ by continuity, and then $\{(u_{x,y},u_{x,y}^*):x\in G,y\in H\}$ is a quasi-basis of $z_{_H}:\overline{\bigcup\limits_{n<m}
{\mathcal{F}}_{_H}(\Lambda_{{n-\frac{1}{2}},m})}
\rightarrow\overline{\bigcup\limits_{n<m}{\mathcal{A}}_{_H}
(\Lambda_{n,m-\frac{1}{2}})}$.

Finally, the uniqueness of the $C^*$-inductive limit (\cite{B.R.Li}) implies that ${\mathcal{F}}_{_H}=\overline{\bigcup\limits_{n<m}
{\mathcal{F}}_{_H}(\Lambda_{{n-\frac{1}{2}},m})}$
and ${\mathcal{A}}_{_{(H,G)}}=
\overline{\bigcup\limits_{n<m}{\mathcal{A}}_{_H}
(\Lambda_{n,m-\frac{1}{2}})}$.
As a result, $\{(u_{x,y},u_{x,y}^*):x\in G,y\in H\}$ is a quasi-basis of $z_{_H}:{\mathcal{F}}_{_H}\rightarrow{\mathcal{A}}_{_{(H,G)}}$.
\end{proof}

From Theorem 3.1, we know $z_{_H}$ is a conditional expectation of index-finite type, which can guarantee that $z_{_H}$ is non-degenerate.

\begin{remark}
(1) For $k,l\in\Bbb{Z}$ with $k\leq l$, $x\in G,y\in H$, set
$$w_{x,y}=\sqrt{|G|}\delta_x(k)\rho_y(l+\frac{1}{2}).$$
One can verify that $\{(w_{x,y},w_{x,y}^*):x\in G,y\in H\}$ is a quasi-basis of $z_{_H}:{\mathcal{F}}_{_H}\rightarrow{\mathcal{A}}_{_{(H,G)}}$.

(2) Let $k,l\in\Bbb{Z}$ with $k>l$, $x\in G,y\in H$, put
$$v_{x,y}=\sqrt{|G|}\delta_x(k)\rho_y(l+\frac{1}{2}),$$
then $\{(v_{x,y},v_{x,y}^*):x\in G,y\in H\}$ is not a quasi-basis of $z_{_H}:{\mathcal{F}}_{_H}\rightarrow{\mathcal{A}}_{_{(H,G)}}$. In fact, one can show that $\{(\nu_{x,y},\nu_{x,y}^*):x\in G,y\in H\}$ is not a quasi-basis of $z_{_H}:{\mathcal{F}}_{_H}(\Lambda_{{\frac{1}{2}},2})
\rightarrow{\mathcal{A}}_{_H}(\Lambda_{1,\frac{3}{2}})$,
where $\nu_{x,y}=\sqrt{|G|}\delta_x(1)\rho_y(\frac{1}{2})$.

Notice that
 \begin{eqnarray*}
    \begin{array}{rcl}
&&\sum\limits_{x\in G}\sum\limits_{y\in H}
\nu_{x,y}z_{_H}\Big[\nu_{x,y}^*\delta_{g_1}(1)\delta_{g_2}(2)
\varrho_{h_1}(\frac{1}{2})\varrho_{h_2}(\frac{3}{2})\Big]\\
&=&|G|\sum\limits_{x\in G}\sum\limits_{y\in H}\nu_{x,y}
z_{_H}\Big[\delta_x(1)\varrho_{y^{-1}}(\frac{1}{2})
\delta_{g_1}(1)\delta_{g_2}(2)
\varrho_{h_1}(\frac{1}{2})\varrho_{h_2}(\frac{3}{2})\Big]\\
&=&|G|\sum\limits_{x\in G}\sum\limits_{y\in H}\nu_{x,y}
z_{_H}\Big[\delta_x(1)\delta_{y^{-1}g_1}(1)\delta_{y^{-1}g_2}(2)
\varrho_{y^{-1}}(\frac{1}{2})
\varrho_{h_1}(\frac{1}{2})\varrho_{h_2}(\frac{3}{2})\Big]\\
&=&|G|\sum\limits_{x\in G}\sum\limits_{y\in H}\nu_{x,y}
z_{_H}\Big[\delta_x(1)\delta_{y^{-1}g_1}(1)\delta_{y^{-1}g_2}(2)
\varrho_{y^{-1}h_1}(\frac{1}{2})\varrho_{h_2}(\frac{3}{2})\Big]\\
&=&|G|\delta_{h_2^{-1}h_1^{-1}g_1}(1)\varrho_{h_1h_2}(\frac{1}{2})
z_{_H}\Big[\delta_{h_2^{-1}h_1^{-1}g_1}(1)\delta_{h_2^{-1}h_1^{-1}g_2}(2)
\varrho_{h_2^{-1}}(\frac{1}{2})
\varrho_{h_2}(\frac{3}{2})\Big]\\
&=&\delta_{h_2^{-1}h_1^{-1}g_1}(1)\varrho_{h_1h_2}(\frac{1}{2})
\Big[\sum\limits_{s\in G}
\delta_s(1)\delta_{sg_1^{-1}g_2}(2)
\varrho_{sg_1^{-1}h_1h_2^{-1}h_1^{-1}g_1s^{-1}}(\frac{1}{2})
\varrho_{sg_1^{-1}h_1h_2h_1^{-1}g_1s^{-1}}
(\frac{3}{2})\Big]\\
&=&\sum\limits_{s\in G}
\delta_{h_2^{-1}h_1^{-1}g_1}(1)\delta_{h_1h_2s}(1)
\delta_{h_1h_2sg_1^{-1}g_2}(2)
\varrho_{h_1h_2sg_1^{-1}h_1h_2^{-1}h_1^{-1}g_1s^{-1}}(\frac{1}{2})
\varrho_{sg_1^{-1}h_1h_2h_1^{-1}g_1s^{-1}}
(\frac{3}{2})\\
&=&\delta_{h_2^{-1}h_1^{-1}g_1}(1)\delta_{h_2^{-1}h_1^{-1}g_2}(2)
\varrho_{h_2^{-1}h_1h_2}(\frac{1}{2})
\varrho_{h_2^{-1}h_1^{-1}h_2h_1h_2}(\frac{3}{2}),
 \end{array}
  \end{eqnarray*}
which can implies that for some $a\in {\mathcal{F}}_{_H}(\Lambda_{{\frac{1}{2}},2})$, we have
\begin{eqnarray*}
    \begin{array}{c}
    \sum\limits_{x\in G}\sum\limits_{y\in H}
    \nu_{x,y}z_{_H}(\nu_{x,y}^*a)\neq a.
      \end{array}
\end{eqnarray*}
\end{remark}

\begin{theorem}
The $C^*$-index of $z_{_H}:{\mathcal{F}}_{_H}\rightarrow{\mathcal{A}}_{_{(H,G)}}$ is $|G||H|I$.
\end{theorem}

\begin{proof}
Since ${\mathrm{Index}}\ z_{_H}$ does not depend on the choice of quasi-basis, then
 \begin{eqnarray*}
    \begin{array}{rcl}
{\mathrm{Index}}\ z_{_H}
&=&\sum\limits_{x\in G}\sum\limits_{y\in H}u_{x,y}u_{x,y}^*\\
&=&|G|\sum\limits_{x\in G}\sum\limits_{y\in H}
\delta_x(1)\rho_y(\frac{3}{2})\delta_x(1)\rho_y^{-1}(\frac{3}{2})\\
&=&|G|\sum\limits_{x\in G}\sum\limits_{y\in H}
\delta_x(1)\delta_x(1)\rho_y(\frac{3}{2})\rho_y^{-1}(\frac{3}{2})\\
&=&|G|\sum\limits_{x\in G}\sum\limits_{y\in H}
\delta_x(1)\rho_e(\frac{3}{2})\\
&=&|G||H|I.
      \end{array}
\end{eqnarray*}
\end{proof}

\begin{remark}
In particular, if $H=G$, then $z_{_G}=\frac{1}{|G|}\sum\limits_{g\in G}(e,g): {\mathcal{F}}\rightarrow{\mathcal{A}}_{_{(G,G)}}$ is a conditional expectation of index-finite type, and ${\mathrm{Index}}\ z_{_G}=|G|^2I$.
\end{remark}


\begin{thebibliography}{}

\bibitem{S.Dop} Doplicher, S., Roberts, J. Fields, statistics and non-abelian gauge group. Comm. Math. Phys., 28: 331-348 (1972)

\bibitem{V.F.R.J} Jones, V.F.R. Subfactors and Knots, CBMS, No. 80, American Mathematical Society Providence, Rhode Island, 1991.

\bibitem{K.Szl} Szlach\'{a}nyi, K., Vecsernyes, P. Quantum symmetry and braided group statistics in $G$-spin models. Comm. Math. Phys., 156: 127-168 (1993)
\bibitem{K.A.Da} Dancer, K.A., Isaac, P.S., Links, J. Representations of the quantum doubles of finite group algebras and spectral parameter dependent solutions of the Yang-Baxter equations. J. Math. Phys., 47: 103511 (2006)
\bibitem{G.Mas} Mason, G. The quantum double of a finite group and its role in conformal field theory, London Mathematical Society Lecture Notes, 212, 405-417, Cambridge Univ. Press, Cambridge, 1995.
\bibitem{P.Ban} B\'{a}ntay, P. Orbifolds and Hopf algebras. Phys. Lett. B., 245: 477-479 (1990)
\bibitem{C.Kas} Kassel, C. Quantum groups, Springer, New York, 1995, GTM 155.
\bibitem{F.Nil} Nill, F., Szlach\'{a}nyi, K. Quantum chains of Hopf algebras with quantum double cosymmetry. Comm. Math. Phys., 187: 159-200 (1997)
\bibitem{L.N.Ji} Jiang, L.N. $C^*$-index of observable algebras in $G$-spin models. Science in China. Ser.A Mathematics, 48: 57-66 (2005)
\bibitem{Qiao} Xin, Q.L., Jiang, L.N. Symmetric structure of field algebra of $G$-spin models determined by a normal subgroup. J. Math. Phys., 55: 091703 (2014)
\bibitem{M.E.Sw} Sweedler, M.E. Hopf algebras, W.A. Benjamin, New York, 1969.
\bibitem{E.Abe} Abe, E. Hopf Algebras, Cambridge Tracts in Mathematics, No. 74, Cambridge Univ. Press, New York, 1980
\bibitem{L.N.Jia} Jiang, L.N. Towards a quantum Galois theory for quantum double algebras of finite groups. Proc. Amer. Math. Soc., 138: 2793-2801 (2010)
\bibitem{B.R.Li} Li, B.R. Operator Algebras, Scientific Press, Beijing, 1998 (in Chinese).
\bibitem{Y.Wat} Watatani, Y. Index for $C^*$-subalgebras, Mem. Am. Math. Soc. No. 424, 1990.
\bibitem{M.PP} Pimsner, M., Popa, S. Entropy and index for subfactors, Ann. Sci. Ecole. Norm. Sup., 19: 57-106 (1986)
\bibitem{H.Ko} Kosaki, H. Extension of Jones' theory on index to arbitrary factors. J. Func. Anal., 66: 123-140 (1986)




 \end{thebibliography}
 \end{document}